\documentclass[12pt, letterpaper]{amsart}
\usepackage{amsfonts,amsthm,mathtools}
\usepackage[all,cmtip]{xy}
\usepackage{fullpage}
\usepackage{amsmath}
\usepackage{amssymb}
\usepackage{enumerate}
\usepackage{tikz}
\usepackage[backref]{hyperref}
\usepackage{cleveref}
\usepackage{verbatim}
\usepackage{cancel}
\usepackage{float}

\newtheorem{lem}{Lemma}[section]
\newtheorem{thm}[lem]{Theorem}

\newtheorem{prop}[lem]{Proposition}
\newtheorem{cor}[lem]{Corollary}
\newtheorem{conj}[lem]{Conjecture}

\theoremstyle{definition}

\newtheorem{definition}[lem]{Definition}

\DeclareMathAlphabet{\curly}{U}{rsfs}{m}{n}

\newcommand{\Q}{\mathbb{Q}}
\newcommand{\C}{\mathbb{C}}

\newcommand{\Z}{\mathbb{Z}}
\newcommand{\F}{\mathbb{F}}

\newcommand{\SL}{\operatorname{SL}}

\mathchardef\mhyphen="2D

\title{Tetragonal modular quotients $X_0^+(N)$}

\author{\sc Petar Orli\'c}
\address{Petar Orli\'c \\
University of Zagreb\\  
Bijeni\v{c}ka Cesta 30 \\
10000 Zagreb\\
Croatia}
\email{petar.orlic@math.hr}

\begin{document}
\begin{abstract}
    In this paper we determine all quotient curves $X_0^+(N)$ whose $\Q$ or $\C$-gonality is equal to $4$. As a consequence, we find several new cases when the modular curve $X_0(N)$ has $\Q$-gonality equal to $8$.
\end{abstract}


\thanks{The author was supported by the QuantiXLie Centre of Excellence, a
  project co-financed by the Croatian Government and European Union
  through the European Regional Development Fund - the Competitiveness
  and Cohesion Operational Programme (Grant KK.01.1.1.01.0004), by the project “Implementation of cutting-edge research and its application as part of the Scientific Center of Excellence for Quantum and Complex Systems, and Representations of Lie Algebras“, PK.1.1.02, European Union, European Regional Development Fund, and by
  the Croatian Science Foundation under the projects
  no. IP-2018-01-1313 and IP-2022-10-5008.}

\maketitle

\section{Introduction}

Let $C$ be a smooth projective curve over a field $k$. The $k$-gonality of $C$, denoted by $\textup{gon}_k C$, is the least degree of a non-constant $k$-rational morphism $f:C\to\mathbb{P}^1$.

Morphisms from modular curves have been extensively studied. Zograf \cite{Zograf1987} gave a lower bound for the $\C$-gonality for any modular curve, linear in terms of the index of the corresponding congruence subgroup. Later, Abramovich \cite{abramovich} and Kim and Sarnak \cite[Appendix 2]{Kim2002} improved the constant in that bound. That lower bound is \Cref{kimsarnakbound} of this paper.

Regarding the curve $X_0(N)$, Ogg \cite{Ogg74} determined all hyperelliptic, Bars \cite{Bars99} determined all bielliptic, Hasegawa and Shimura \cite{HasegawaShimura_trig} determined all trigonal curves $X_0(N)$ over $\C$ and $\Q$, and Jeon and Park \cite{JeonPark05} determined all tetragonal curves $X_0(N)$ over $\C$. Recently, Najman and Orlić \cite{NajmanOrlic22} determined all curves $X_0(N)$ with $\Q$-gonality equal to $4,5,$ or $6$, and for many of those curves also determined their $\C$-gonality.

Much progress has been made in studying the $\Q$-gonality of the modular quotients of $X_0(N)$ as well. Furumoto and Hasegawa \cite{FurumotoHasegawa1999} determined all hyperelliptic quotients of $X_0(N)$, Hasegawa and Shimura \cite{HasegawaShimura1999,HasegawaShimura2000,HasegawaShimura2006} determined all trigonal quotients of $X_0(N)$ over $\C$ and partially determined the $\Q$-trigonal quotients (more precisely, they solved all cases when the genus is not equal to $4$). Furthermore, Bars, Gonzalez and Kamel \cite{BARS2020380} determined all bielliptic quotients of $X_0(N)$ for squarefree levels, Jeon \cite{JEON2018319} determined all bielliptic curves $X_0^+(N)$, and Bars, Kamel and Schweizer \cite{bars22biellipticquotients} determined all bielliptic quotients of $X_0(N)$ for non-squarefree levels.

The next logical step is determining all tetragonal quotients of $X_0(N)$. In this paper we will study the $\Q$ and $\C$-gonality of the quotient curve $X_0^+(N)=X_0(N)/w_N$ ($w_N$ being the Atkin-Lehner involution).

One of the reasons we study the gonality of $X_0^+(N)$, apart from being an interesting question in itself, is that it can help us to determine the gonality of $X_0(N)$. The reason for that is, of course, that we have a natural rational quotient map $X_0(N)\to X_0^+(N)$ of degree $2$. Therefore, any map $X_0^+(N)\to\mathbb{P}^1$ of degree $d$ induces a map $X_0(N)\to\mathbb{P}^1$ of degree $2d$ defined over the same base field.

Also, the existence of rational degree $d$ maps to $\mathbb{P}^1$ is closely linked to the problem of determining whether that curve has infinitely many points of degree $d$, as can for example be seen in \cite{AbramovichHarris91, HarrisSilverman91}, and, more recently, \cite[Theorem 1.4]{KadetsVogt}.

For a field $K$, a $K$-curve is an elliptic curve defined over some finite separable extension of $K$ which is isogenous over $\overline{K}$ to all its $\textup{Gal}(\overline{K}/K)$ conjugates \cite[Page 81]{Elkies2004}. The $\Q$-curves are most studied of the $K$-curves. There is a number of papers which use Frey $\Q$-curves defined over quadratic fields to solve Diophantine equations. Some of the more recent ones are \cite{Pacetti2023, Pacetti2022}.

Non-cuspidal rational points on the curve $X_0^+(N)$ correspond to a pair of $\Q$-curves of degree $N$ defined over a quadratic field \cite[Section 2]{NajmanVukorepa}. Similarly, for a number field $K$, $K$-rational points on the curve $X_0^+(N)$ correspond to a pair of $K$-curves of degree $N$ defined over a quadratic extension of $K$. Therefore, determining the $\Q$-gonality of curves $X_0^+(N)$ could be useful in determining whether there are infinitely many $K$-curves of a certain degree defined over a quadratic extension of $K$.

Our main results are the following theorems (though the first theorem was already mostly proved by Hasegawa and Shimura).

\begin{thm}\label{trigonalthm}
    The curve $X_0^+(N)$ has $\Q$-gonality equal to $3$ if and only if
    \begin{align*}
        N\in\{ &58,76,84,86,88,93,96,97,99,100,109,113,115,116,122,127,128,129,135,\\
        &137,139,146,147,149,151,155,159,162,164,169,179,181,215,227,239\} .
    \end{align*}
\end{thm}

\begin{thm}\label{CtrigonalQtetragonalthm}
    The curve $X_0^+(N)$ has $\Q$-gonality equal to $4$ and $\C$-gonality equal to $3$ if and only if
    $$N\in\{70,82,90,108,117,161,173,199,251,311\}.$$
\end{thm}

\begin{thm}\label{tetragonalthm}
    The curve $X_0^+(N)$ has $\Q$-gonality and $\C$-gonality equal to $4$ if and only if
    \begin{align*}
        N\in\{ &78,102,105,106,110,112,114,118,120,123,124,126,133,134,136,138,140,\\
        &141,142,144,145,148,152,156,157,158,160,163,165,166,171,175,176,177,\\
        &183,184,185,188,192,193,194,195,197,200,203,205,206,207,209,211,213,\\
        &221,223,224,229,241,257,263,269,279,281,284,287,299,359\}.
    \end{align*}
\end{thm}

\begin{thm}\label{Ctetragonalthm}
    The curve $X_0^+(N)$ has $\Q$-gonality greater than $4$ and $\C$-gonality equal to $4$ if and only if $N\in\{243,271\}$.
\end{thm}

As a consequence, we were able to determine the $\Q$-gonality of $X_0(N)$ for several new levels $N$, not previously solved in \cite{NajmanOrlic22}.

It is perhaps surprising that we were able to determine all tetragonal curves $X_0^+(N)$ using mostly previously known methods. One of the reasons for that is that in many cases we can get the degree $4$ map from a degree $2$ quotient map to $X_0(N)/\left<w_d,w_N\right>$, as presented in \Cref{quotientmapprop}. If one would, for example, search for all pentagonal curves $X_0^+(N)$, there is no such natural source of degree $5$ maps.

A lot of the results in this paper rely on \texttt{Magma} computations \cite{magma}. The codes that verify all computations in this paper can be found on
\begin{center}
    \url{https://github.com/orlic1/gonality_X0_quotients}.
\end{center}
All computations were performed on the Euler server at the Department of Mathematics, University of Zagreb with a Intel Xeon W-2133 CPU running at 3.60GHz and with
64 GB of RAM.

\section{Acknowledgements}

Many thanks to Nikola Adžaga and Philippe Michaud-Jacobs for permission to use their \texttt{Magma} codes as templates. I am also grateful to Peter Bruin, Maarten Derickx, and Filip Najman for their helpful comments and suggestions and to the anonymous referee for useful comments that have greatly improved the exposition.

\section{Results}
In this section we list the known results that will be used to determine the $\Q$-gonality of $X_0^+(N)$. We first mention two obvious lower bounds for a curve $C$ defined over $\Q$:
\begin{align*}
    \textup{gon}_\C(C)&\leq\textup{gon}_\Q(C),\\
    \textup{gon}_{\F_p}(C)&\leq\textup{gon}_\Q(C).
\end{align*}
Here $p$ is a prime of good reduction for $C$.

\subsection{$\F_p$-gonality}\label{Fpsection}

An important tool for determining $\Q$-gonalities of modular curves $X_0(N)$ is Ogg's inequality \cite[Theorem 3.1]{Ogg74}, stated in simpler form in \cite[Lemma 3.1]{HasegawaShimura_trig}.

\begin{lem}[Ogg]\label{lemmaogg}
    Let $p$ be a prime not dividing $N$. Then the number of $\F_{p^2}$ points on $X_0(N)$ is at least
    $$L_p(N):=\frac{p-1}{12}\psi(N)+2^{\omega(N)}.$$
    Here $\psi(N)$ is the index of the congruence subgroup $\Gamma_0(N)$, equal to $N\prod_{q\mid N}(1+\frac{1}{q})$, and $\omega(N)$ is the number of different prime divisors of $N$.
\end{lem}
\begin{lem}\cite[Lemma 3.5]{NajmanOrlic22}\label{Fp2points}
    Let $C$ be a curve, $p$ a prime of good reduction for $C$, and $q$ a power of $p$. Suppose $\#C(\F_q)>d(q+1)$. Then $\textup{gon}_\Q(C)>d$.
\end{lem}

From \Cref{lemmaogg} and \Cref{Fp2points} (applied for $q=p^2$) we can get a lower bound on the $\Q$-gonality of $X_0^+(N)$. Namely, if the quotient curve $X_0^+(N)$ is tetragonal, then we have a rational composition map $X_0(N)\to X_0^+(N)\to\mathbb{P}^1$ of degree $8$. Therefore, we must have
\begin{align}\label{ogg}
    L_p(N)&\leq8(p^2+1)
\end{align}
for all primes $p\nmid N$. However, similarily as in \cite[Lemma 3.2]{HasegawaShimura_trig}, we get that for $N\geq456$ there exists a prime $p$ for which this inequality does not hold. This means that we have eliminated all but finitely many levels $N$. This inequality can also be used to eliminate
\begin{align*}
    N\in\{&255,260,266,276,280,282,285,286,290,292,294,296,304,306,308,310,312,314,315,\\
    &316,318,320,322,324,326,327,328,330,332,333,334,336,338,339,340,342,344,345,\\
    &346,348,350,351,352,354,356,357,358,360,362-366,368,369,370,372,374,375,\\
    &376,378,380,381,382,384-388,390,392-396,398,399,400,402-408,410-418,\\
    &422-430,432,434-438,440,441,442,444-448,450-455\}.
\end{align*}

We can also use \Cref{Fp2points} directly to deal with some cases.

\begin{prop}
    The curve $X_0^+(N)$ is not tetragonal over $\Q$ for the following values of $N$:

\begin{center}
\begin{tabular}{|c|c|c||c|c|c||c|c|c|}
\hline
$N$ & $p$ & $\#X_0^+(N)(\F_{p^2})$ & $N$ & $p$ & $\#X_0^+(N)(\F_{p^2})$ & $N$ & $p$ & $\#X_0^+(N)(\F_{p^2})$\\
    \hline
    
$268$ & $3$ & $46$ & $272$ & $3$ & $42$ & $273$ & $2$ & $26$\\
$274$ & $3$ & $48$ & $288$ & $5$ & $116$ & $291$ & $2$ & $21$\\
$297$ & $2$ & $27$ & $298$ & $3$ & $45$ & $301$ & $2$ & $21$\\
$305$ & $2$ & $24$ & $309$ & $2$ & $23$ & $323$ & $2$ & $23$\\
$325$ & $2$ & $23$ & $341$ & $2$ & $25$ & $343$ & $3$ & $43$\\
$347$ & $2$ & $21$ & $349$ & $2$ & $22$ & $353$ & $2$ & $22$\\
$355$ & $2$ & $22$ & $361$ & $2$ & $22$ & $367$ & $2$ & $46$\\
$371$ & $2$ & $25$ & $373$ & $2$ & $21$ & $377$ & $2$ & $24$\\
$379$ & $2$ & $22$ & $389$ & $2$ & $24$ & $391$ & $2$ & $24$\\
$397$ & $3$ & $41$ & $401$ & $2$ & $24$ & $409$ & $2$ & $25$\\
$419$ & $2$ & $23$ & $421$ & $2$ & $25$ & $433$ & $2$ & $23$\\
$439$ & $2$ & $22$ & $443$ & $2$ & $25$ & $449$ & $2$ & $26$\\

    \hline
\end{tabular}
\end{center}
\end{prop}

\begin{proof}
    Using \texttt{Magma}, we calculate the number of $\F_{p^2}$ points on $X_0^+(N)$. It is now easy to check that $\#X_0^+(N)(\F_{p^2})>4(p^2+1)$ in all these cases.
\end{proof}

We continue with computing the $\F_p$-gonality of $X_0^+(N)$. This is a finite task and can be done by checking the Riemann-Roch spaces of degree $d$ effective $\F_p$-rational divisors. We can also use \cite[Lemma 3.1]{NajmanOrlic22} to reduce the number of divisors that need to be checked and quicken the computation.

\begin{prop}\label{Fp_gonality}
The $\F_p$-gonality of $X_0^+(N)$ is bounded from below for the following values of $N$:

\begin{center}
\begin{tabular}{|c|c|c||c|c|c||c|c|c||c|c|c|c|}
\hline
$N$ & $p$ & $\textup{gon}_{\F_p}\geq$ & $N$ & $p$ & $\textup{gon}_{\F_p}\geq$ & $N$ & $p$ & $\textup{gon}_{\F_p}\geq$ & $N$ & $p$ & $\textup{gon}_{\F_p}\geq$\\
    \hline
    
$70$ & $11$ & $4$ & $82$ & $3$ & $4$ & $90$ & $11$ & $4$ & $108$ & $5$ & $4$\\ $117$ & $7$ & $4$ & $130$ & $3$ & $5$ & $132$ & $5$ & $5$ & $150$ & $7$ & $5$\\
$154$ & $3$ & $5$ & $161$ & $2$ & $4$ & $168$ & $5$ & $5$ & $170$ & $3$ & $5$\\ $172$ & $3$ & $5$ & $173$ & $5$ & $4$ & $174$ & $5$ & $5$ & $178$ & $3$ & $5$\\ $180$ & $7$ & $5$ & $182$ & $3$ & $5$ & $187$ & $3$ & $5$ & $189$ & $2$ & $5$\\ $196$ & $3$ & $5$ & $198$ & $5$ & $5$ & $199$ & $5$ & $4$ & $201$ & $2$ & $5$\\ $202$ & $3$ & $5$ & $204$ & $5$ & $5$ & $208$ & $3$ & $5$ & $212$ & $3$ & $5$\\ $216$ & $5$ & $5$ & $217$ & $2$ & $5$ & $218$ & $3$ & $5$ & $219$ & $2$ & $5$\\ $225$ & $2$ & $5$ & $226$ & $3$ & $5$ & $228$ & $5$ & $5$ & $230$ & $3$ & $5$\\ $231$ & $2$ & $5$ & $232$ & $3$ & $5$ & $233$ & $2$ & $5$ & $234$ & $5$ & $5$\\ $235$ & $2$ & $5$ & $237$ & $2$ & $5$ & $240$ & $7$ & $5$ & $242$ & $3$ & $5$\\ $243$ & $7$ & $5$ & $244$ & $3$ & $5$ & $245$ & $2$ & $5$ & $247$ & $3$ & $5$\\ $250$ & $3$ & $5$ & $251$ & $2$ & $4$ & $253$ & $2$ & $5$ & $256$ & $3$ & $5$\\ $259$ & $2$ & $5$ & $261$ & $2$ & $5$ & $265$ & $2$ & $5$ & $271$ & $3$ & $5$\\ $275$ & $2$ & $5$ & $277$ & $2$ & $5$ & $283$ & $2$ & $5$ & $289$ & $2$ & $5$\\ $293$ & $2$ & $5$ & $307$ & $2$ & $5$ & $311$ & $2$ & $4$ & $313$ & $2$ & $5$\\ $317$ & $2$ & $5$ & $319$ & $2$ & $5$ & $331$ & $2$ & $5$ & $335$ & $2$ & $5$\\ $337$ & $2$ & $5$ & $383$ & $2$ & $5$ & & & & & & \\

    \hline
\end{tabular}
\end{center}
\end{prop}

\begin{proof}
    Using \texttt{Magma}, we compute that there are no functions of degree $<d$ in $\F_p(X_0^+(N))$. We are able to reduce the number of divisors that need to be checked by noting the following: If there exists a function $f$ over a field $k$ of a certain degree and if $c\in k$, then the function $g(x):=\frac{1}{f(x)-c}$ has the same degree and its polar divisor contains a $k$-rational point.
\end{proof}

In the following propositions we will use Poonen's \cite[Proposition A.1.]{Poonen2007}, stated below.

\begin{prop}[Poonen]\label{poonen}
    Let $X$ be a curve of genus $g$ over a field $k$.
    \begin{enumerate}[(i)]
        \item If $L$ is a field extension of $k$, then $\textup{gon}_L(X)\leq \textup{gon}_k(X)$.
        \item If $k$ is algebraically closed and $L$ is a field extension of $k$, then $\textup{gon}_L(X)=\textup{gon}_k(X)$.
        \item If $g\geq2$, then $\textup{gon}_k(X)\leq 2g-2$.
        \item If $g\geq2$ and $X(k)\neq\emptyset$, then $\textup{gon}_k(X)\leq g$.
        \item If $k$ is algebraically closed, then $\textup{gon}_k(X)\leq\frac{g+3}{2}$.
        \item If $\pi:X\to Y$ is a dominant $k$-rational map, then $\textup{gon}_k(X)\leq \deg \pi\cdot\textup{gon}_k(Y)$.
        \item If $\pi:X\to Y$ is a dominant $k$-rational map, then $\textup{gon}_k(X)\geq\textup{gon}_k(Y)$.
    \end{enumerate}
\end{prop}

Since all modular curves $X_0(N)$ and their quotients have at least one rational cusp, this result implies that their $\Q$-gonality is bounded from above by their genus.

\begin{prop}
The $\F_p$-gonality of $X_0^+(N)$ is at least $5$ for the following values of $N$:

\begin{center}
\begin{tabular}{|c|c|c|}
\hline
$N$ & $p$ & $d$ \\
    \hline
    
$246$ & $5$ & $3$\\
$264$ & $5$ & $3$\\

    \hline
\end{tabular}
\end{center}
\end{prop}

\begin{proof}
    Using \texttt{Magma}, we compute that there are no degree $4$ functions in $F_p(X_0(N)/\left<w_d,w_N\right>)$. Since there is a degree $2$ quotient map $X_0^+(N)\to X_0(N)/\left<w_d,w_N\right>$, \Cref{poonen}(vii) tells us that $\F_p$-gonality of $X_0^+(N)$ is $\geq5$.
\end{proof}

\subsection{Castelnuovo-Severi inequality}\label{CSsection}

This is a very useful tool for producing a lower bound on the gonality (see \cite[Theorem 3.11.3]{Stichtenoth09} for a proof).
\begin{prop}[Castelnuovo-Severi inequality] 
\label{tm:CS}
Let $k$ be a perfect field, and let $X,\ Y, \ Z$ be curves over $k$ with genera $g(X), g(Y), g(Z)$. Let non-constant morphisms $\pi_Y:X\rightarrow Y$ and $\pi_Z:X\rightarrow Z$ over $k$ be given, and let their degrees be $m$ and $n$, respectively. Assume that there is no morphism $X\rightarrow X'$ of degree $>1$ through which both $\pi_Y$ and $\pi_Z$ factor. Then the following inequality holds:
\begin{equation} \label{eq:CS}
g(X)\leq m \cdot g(Y)+n\cdot g(Z) +(m-1)(n-1).
\end{equation}
\end{prop}

Since $\C$ and $\Q$ are both perfect fields, we can use Castelnuovo-Severi inequality to get lower bounds on both $\C$ and $\Q$-gonalities.

As we can see, the assumption of Castelnuovo-Severi inequality is that there is no morphism $X\to X'$ defined over $\overline{k}$ through which both $\pi_Y$ and $\pi_Z$ factor. However, recently Khawaja and Siksek have recently shown \cite[Theorem 14]{KhawajaSiksek2023} that we can weaken this assumption. Namely, that there is no morphism $X\to X'$ defined over $k$ through which both $\pi_Y$ and $\pi_Z$ factor.

\begin{prop}
The $\C$-gonality of $X_0^+(N)$ is at least $5$ for the following values of $N$:

\begin{center}
\begin{tabular}{|c|c|c|c||c|c|c|c|}
\hline
$N$ & $g(X_0^+(N))$ & $d$ & $g(X_0(N)/\left<w_d,w_N\right>)$ & $N$ & $g(X_0^+(N))$ & $d$ & $g(X_0(N)/\left<d,N\right>)$\\
    \hline
    
$186$ & $12$ & $3$ & $3$ & $190$ & $13$ & $2$ & $3$\\
$210$ & $19$ & $6$ & $6$ & $214$ & $12$ & $4$ & $4$\\
$220$ & $14$ & $4$ & $4$ & $222$ & $15$ & $2$ & $4$\\
$236$ & $10$ & $4$ & $3$ & $238$ & $15$ & $3$ & $3$\\
$248$ & $11$ & $8$ & $3$ & $249$ & $11$ & $3$ & $3$\\
$252$ & $17$ & $4$ & $5$ & $254$ & $12$ & $2$ & $4$\\
$258$ & $19$ & $3$ & $7$ & $262$ & $15$ & $2$ & $4$\\
$266$ & $14$ & $14$ & $5$ & $267$ & $13$ & $3$ & $4$\\
$270$ & $19$ & $2$ & $7$ & $276$ & $18$ & $12$ & $5$\\
$278$ & $14$ & $2$ & $5$ & $282$ & $21$ & $6$ & $6$\\
$286$ & $17$ & $2$ & $4$ & $295$ & $11$ & $5$ & $3$\\
$300$ & $19$ & $4$ & $7$ & $302$ & $16$ & $2$ & $5$\\
$303$ & $12$ & $3$ & $3$ & $310$ & $21$ & $3$ & $8$\\
$312$ & $23$ & $8$ & $8$ & $316$ & $17$ & $4$ & $5$\\
$318$ & $23$ & $2$ & $7$ & $321$ & $13$ & $3$ & $4$\\
$329$ & $10$ & $7$ & $3$ & $330$ & $31$ & $3$ & $13$\\
$420$ & $39$ & $3$ & $17$ & & & &\\  

    \hline
\end{tabular}
\end{center}
\end{prop}

\begin{proof}
    \cite[Theorem 1, Theorem 2, Proposition 1]{HasegawaShimura1999} tell us that the quotient curves $X_0^+(N)$ are not hyperelliptic nor trigonal over $\C$ for these levels $N$. 
    
    We have a degree $2$ quotient map from $X_0^+(N)$ to $X_0(N)/\left<w_d,w_N\right>$. If there existed a degree $4$ map from $X_0^+(N)$ to $\mathbb{P}^1$, then we apply the Castelnuovo-Severi inequality to this hypothetical degree $4$ map and the degree $2$ quotient map. Since
    $g(X_0^+(N))>4\cdot0+2\cdot g(X_0(N)/\left<w_d,w_N\right>)+3\cdot1,$ we conclude that the degree $4$ map would have to factor through the quotient map $X_0^+(N)\to X_0(N)/\left<w_d,w_N\right>$ and the curve $X_0(N)/\left<w_d,w_N\right>$ would need to be elliptic or hyperelliptic. However, this is impossible by \cite[Theorem 3, Theorem 4]{FurumotoHasegawa1999}.
\end{proof}

\subsection{Degree $4$ maps to $\mathbb{P}^1$}\label{degree4mapsection}

In \Cref{Fpsection} and \Cref{CSsection}, we were proving that the curve $X_0^+(N)$ is not $\Q$-tetragonal by arguing that $\C$ or $\F_p$-gonalities are too large. Now we find degree $4$ rational morphisms from $X_0^+(N)$ to $\mathbb{P}^1$ for all levels $N$ listed in \Cref{tetragonalthm}.

\begin{prop}\label{quotientmapprop}
    There exists a degree $4$ rational morphism from $X_0^+(N)$ to $\mathbb{P}^1$ for the following values of $N$:

\begin{center}
\begin{tabular}{|c|c|c||c|c|c||c|c|c|}
\hline
$N$ & $d$ & $g(X_0(N)/\left<w_d,w_N\right>)$ & $N$ & $d$ & $g(X_0(N)/\left<w_d,w_N\right>)$ & $N$ & $d$ & $g(X_0(N)/\left<w_d,w_N\right>)$\\
    \hline
    
$78$ & $2$ & $1$ & $102$ & $2$ & $2$ & $105$ & $3$ & $1$ \\
$106$ & $2$ & $2$ & $110$ & $2$ & $1$ & $112$ & $7$ & $2$\\
$114$ & $3$ & $2$ & $118$ & $2$ & $1$ & $120$ & $8$ & $2$\\
$123$ & $3$ & $1$ & $124$ & $4$ & $1$ & $126$ & $2$ & $2$\\
$133$ & $7$ & $2$ & $134$ & $2$ & $2$ & $138$ & $6$ & $2$\\
$140$ & $4$ & $2$ & $141$ & $3$ & $1$ & $142$ & $2$ & $1$\\
$145$ & $5$ & $1$ & $153$ & $9$ & $2$ & $156$ & $4$ & $2$\\
$158$ & $2$ & $2$ & $165$ & $11$ & $3$ & $166$ & $2$ & $2$\\
$177$ & $3$ & $2$ & $184$ & $8$ & $2$ & $188$ & $4$ & $1$\\
$195$ & $5$ & $3$ & $205$ & $5$ & $2$ & $206$ & $2$ & $2$\\
$207$ & $9$ & $3$ & $209$ & $11$ & $2$ & $213$ & $3$ & $2$\\
$221$ & $13$ & $2$ & $279$ & $9$ & $5$ & $284$ & $4$ & $2$\\
$287$ & $7$ & $2$ & $299$ & $13$ & $2$ & & &\\

    \hline
\end{tabular}
\end{center}
    
\end{prop}

\begin{proof}
    In all these cases the quotient $X_0(N)/\left<w_d,w_N\right>$ is elliptic or hyperelliptic and the degree $4$ rational map can be realised as a composition map $X_0^+(N)\to X_0(N)/\left<w_d,w_N\right>\to \mathbb{P}^1$. If the genus of the quotient is higher than $2$, we can use \cite[Theorem 3, Theorem 4]{FurumotoHasegawa1999} to prove its hyperellipticity.
\end{proof}

\begin{prop}\label{genus6gonalmap}
    There exists a degree $4$ rational map from $X_0^+(N)$ to $\mathbb{P}^1$ for
    $$N\in\{136,152,163,183,197,203,211,223,269,359\}.$$
\end{prop}

\begin{proof}
    The curve $X_0^+(N)$ is of genus $6$ and we can use the \texttt{Magma} function Genus6GonalMap(C) to get the desired map.

    It is good to mention here that this function always returns a morphism of degree $\leq4$ to $\mathbb{P}^1$ since all genus $6$ curves have $\C$-gonality at most $4$ by \Cref{poonen} (v). In the cases listed in this proposition, that morphism will be defined over $\Q$. However, not all genus $6$ curves are $\Q$-tetragonal, $X_0^+(243)$ for example, as we have seen in \Cref{Fp_gonality}.
\end{proof}

\begin{prop}
    There exists a degree $4$ rational map from $X_0^+(N)$ to $\mathbb{P}^1$ for
    $$N\in\{144,148,157,171,175,176,185,193,194,200,263\}.$$
\end{prop}

\begin{proof}
    We find a function of degree $4$ by searching the Riemann-Roch spaces of divisors of the form $P_1+P_2+P_3+P_4$, where $P_i\in X_0^+(N)(\Q)$.
\end{proof}

\begin{prop}
    There exists a degree $4$ rational map from $X_0^+(N)$ to $\mathbb{P}^1$ for
    $$N\in\{160,192,224,229,241,257,281\}.$$
\end{prop}

\begin{proof}
    In these cases we were not able to find a degree $4$ function whose polar divisor is supported on rational points so we had to search for quadratic points.

    We searched for quadratic points by intersecting the curve $X_0^+(N)$ with hyperplanes of the form
    $$b_0x_0+\ldots+b_kx_k=0,$$
    where $b_0,\ldots,b_k\in \Z$ are coprime and chosen up to a certain bound, a similar idea as in \cite[Section 3.2]{Box19}. We can improve this by noting that, in a quadratic point $(x_0,\ldots,x_k)$, already its first three coordinates must be linearly dependent over $\Q$. Therefore, it is enough to check the hyperplanes
    $$b_0x_0+b_1x_1+b_2x_2=0.$$
    In all of these cases we found a function of degree $4$ lying in the Riemann-Roch space of a divisor of the form $P_1+P_2+Q+\sigma(Q)$, where $P_1,P_2\in X_0^+(N)(\Q)$, and $Q$ is one of the quadratic points we found.
\end{proof}

It is worth mentioning here that the running time was $\sim20$ minutes for $N=192$ and $\sim4.5$ hours for $N=224$. Most of that time was spent on searching for points. Other computations in this paper were much faster.

\subsection{$\C$-gonalities}\label{bettisection}

In this section we will determine the cases $\textup{gon}_\C(X_0^+(N))=4$. The only levels $N$ we need to look at are those for which $X_0^+(N)$ is not hyperelliptic, trigonal, nor $\Q$-tetragonal (since for any curve $C$ we have $\textup{gon}_\C(C)\leq \textup{gon}_\Q(C)$) and $g(X_0^+(N))\leq9$. The last condition is due to the following corollary of the Tower Theorem \cite[Theorem 2.1]{NguyenSaito}.

\begin{thm}[The Tower Theorem]
Let $C$ be a curve defined over a perfect field $k$ such that $C(k)\neq0$ and let $f:C\to \mathbb{P}^1$ be a non-constant morphism over $\overline{k}$ of degree $d$. Then there exists a curve $C'$ defined over $k$ and a non-constant morphism $C\to C'$ defined over $k$ of degree $d'$ dividing $d$ such that the genus of $C'$ is $\leq (\frac{d}{d'}-1)^2$.
\end{thm}

\begin{cor}\cite[Corollary 4.6. (ii)]{NajmanOrlic22}
    Let $C$ be a curve defined over $\Q$ with $\textup{gon}_\C (C)=4$ and $g(C)\geq10$ and such that $C(\Q)\neq\emptyset$. Then $\textup{gon}_\Q (C)=4$.
\end{cor}

In order to bound the number of levels we need to check, we can use the theorem by Kim and Sarnak \cite[Appendix 2]{Kim2002}, mentioned in the Introduction. It was more clearly stated in the paper by Jeon, Kim and Park \cite{JeonKimPark06}.

\begin{thm}\cite[Theorem 1.2.]{JeonKimPark06}\label{kimsarnakbound}
    Let $X_\Gamma$ be the algebraic curve corresponding to a congruence subgroup $\Gamma\subseteq \SL_2(\Z)$ of index
    $$D_\Gamma=[\SL_2(\Z):\pm\Gamma].$$
    If $X_\Gamma$ is $d$-gonal, then $D_\Gamma\leq \frac{12000}{119}d$.
\end{thm}

\begin{cor}
    $X_0^+(N)$ is not tetragonal for $N\geq807$.
\end{cor}

\begin{proof}
    Suppose $X_0^+(N)$ is tetragonal. Then $X_0(N)$ has a degree $8$ map to $\mathbb{P}^1$. Since $-I\in\Gamma_0(N)$, we have that $\psi(N)=D_{\Gamma_0(N)}\leq\frac{12000}{119}\cdot8$ (here $\psi(N)=N\prod_{q\mid N}(1+\frac{1}{q})$, as mentioned in \Cref{lemmaogg}).
\end{proof}

Therefore, any curve $X_0^+(N)$ which is $\C$-tetragonal, but not $\Q$-tetragonal must be of genus $\leq9$ and level $N\leq806$. This leaves us with reasonably many cases to be solved manually. The only levels $N$ we need to check are in the table below.

\begin{table}[ht]
\centering
\begin{tabular}{|c|c||c|c||c|c||c|c|}
\hline
$N$ & $g(X_0^+(N))$ & $N$ & $g(X_0^+(N))$ & $N$ & $g(X_0^+(N))$ & $N$ & $g(X_0^+(N))$\\
    \hline

    $130$ & $8$ & $132$ & $8$ & $150$ & $8$ & $154$ & $9$\\
    $170$ & $9$ & $172$ & $9$ & $178$ & $9$ & $187$ & $7$\\
    $189$ & $7$ & $196$ & $7$ & $201$ & $8$ & $217$ & $8$\\
    $219$ & $8$ & $225$ & $8$ & $231$ & $9$ & $233$ & $7$\\
    $242$ & $9$ & $243$ & $7$ & $245$ & $8$ & $247$ & $8$\\
    $256$ & $9$ & $259$ & $8$ & $271$ & $6$ & $275$ & $9$\\
    $283$ & $9$ & $289$ & $7$ & $293$ & $8$ & $335$ & $8$\\
    $341$ & $9$ & $361$ & $9$ & $383$ & $8$ & $419$ & $9$\\
    $431$ & $9$ & $479$ & $8$ & & & &\\ 

    \hline
\end{tabular}
\vspace{5mm}
\caption{Levels $N<915$ for which the curve $X_0^+(N)$ is not $\Q$-tetragonal and $g(X_0^+(N))\leq9$.}
\label{tab:main}
\end{table}

By \Cref{poonen}(v), we immediately see that for $N=271$ there exists a degree $4$ morphism since the genus is $6$. For the other cases we will use graded Betti numbers $\beta_{i,j}$. We will follow the notation in \cite[Section 1.]{JeonPark05}. The results we mention can be found there and in \cite[Section 3.1.]{NajmanOrlic22}.

\begin{definition}
    For a curve $X$ and divisor $D$ of degree $d$, $g_d^r$ is a subspace $V$ of the Riemann-Roch space $L(D)$ such that $\dim V=r+1$.
\end{definition}

Therefore, we want to determine whether $X_0^+(N)$ has a $g_4^1$. Green's conjecture relates graded Betti numbers $\beta_{i,j}$ with the existence of $g_d^r$.

\begin{conj}[Green, \cite{Green84}]
    Let $X$ be a curve of genus $g$. Then $\beta_{p,2}\neq 0$ if and only if there exists a divisor $D$ on $X$ of degree $d$ such that a subspace $g_d^r$ of $L(D)$ satisfies $d\leq g-1$, $r=\ell(D)-1\geq1$, and $d-2r\leq p$.
\end{conj}

The "if" part of this conjecture has been proven in the same paper.

\begin{thm}[Green and Lazarsfeld, Appendix to \cite{Green84}]\label{thmGreenLazarsfeld}
    Let $X$ be a curve of genus $g$. If $\beta_{p,2}=0$, then there does not exist a divisor $D$ on $X$ of degree $d$ such that a subspace $g_d^r$ of $L(D)$ satisfies $d\leq g-1$, $r\geq1$, and $d-2r\leq p$.
\end{thm}

\begin{cor}
    Let $X$ be a curve of genus $g\geq5$ with $\beta_{2,2}=0$. Suppose that $X$ is neither hyperelliptic nor trigonal. Then $\textup{gon}_\C(X)\geq5$.
\end{cor}

\begin{proof}
    Suppose that $\textup{gon}_\C(X)=4$. Then there is a degree $4$ morphism $f:X\to\mathbb{P}^1$. We have $\textup{div}(f)=P-Q$, where $P$ is a zero divisor and $Q$ is a polar divisor. This means that $\ell(Q)\geq2$. 
    
    Suppose that $\ell(Q)=2$. Since $d=\deg Q=4$, this implies the existence of $g_4^1$. However, from the assumptions we have $\beta_{2,2}=0$ and by \Cref{thmGreenLazarsfeld} (by plugging in $d=4$, $r=1$, $p=2$) this is impossible. Therefore, we must have $\ell(Q)\geq3$ and there exists another morphism $g:X\to\mathbb{P}^1$ such that $\textup{div}(g)=R-Q$ and $1,f,g$ are linearly independent.

    Let us fix $P_0\in X(\overline{\Q})$. Then the morphisms $f':=f-f(P_0)$ and $g':=g-g(P_0)$ are in $L(D)$ and have a common zero $P_0$. We now have
    $$\textup{div}(f')=P_0+P'-Q, \ \textup{div}(g')=P_0+R'-Q$$
    for some effective degree $3$ divisors $P',R'$. Also, $P'\neq R'$ since the morphisms $1,f,g$ are linearly independent. Therefore, $\textup{div}(f'/g')=P'-R'$ and $f'/g'$ is a non-constant morphism from $X$ to $\mathbb{P}^1$ of degree $\leq3$, a contradiction.
\end{proof}

\begin{cor}
    The curve $X_0^+(N)$ is not tetragonal for all $N$ in \Cref{tab:main} except $243,271$.
\end{cor}

\begin{proof}
    For all these curves we compute $\beta_{2,2}=0$ using \texttt{Magma}. The computation time for $N=361$ was around $1$ hour, the other computations were much faster and took less than $10$ minutes.
\end{proof}

To prove the next result, we need to introduce the Clifford index and Clifford dimension.

\begin{definition}
    For a curve $X$, let $D$ be a divisor on $X$ and $K$ a canonical divisor on $X$. The Clifford index of $D$ is the integer
    $$\textup{Cliff}(D):=\deg D-2(\ell(D)-1),$$
    and the Clifford index of $X$ is
    $$\textup{Cliff}(X):=\min\{\textup{Cliff}(D) \mid \ell(D)\geq2, \ell(K-D)\geq2\}.$$
    The Clifford dimension of $X$ is defined to be
    $$\textup{CD}(X):=\min \{\ell(D)-1 \mid \ell(D)\geq2, \ell(K-D)\geq2, \textup{Cliff}(D)=\textup{Cliff}(X)\}.$$
    For every such divisor $D$ that achieves the minimum we say that it computes the Clifford dimension.
    
\end{definition}

The Clifford index gives bounds for the $\C$-gonality of $X$ \cite{CoppensMartens91}:
$$\textup{Cliff}(X)+2\leq \textup{gon}_\C(X)\leq\textup{Cliff}(X)+3.$$

We can see that $\textup{CD}(X)\geq1$ immediately from the definition. In most cases we have $\textup{CD}(X)=1$. It is a classical result that $\textup{CD}(X)=2$ if and only if $X$ is a smoooth plane curve of degree $\geq5$ \cite[Page 309]{JeonPark05}. Martens \cite{Martens82} proved that $\textup{CD}(X)=3$ if and only if $X$ is a complete intersection of two irreducible cubic surfaces in $\mathbb{P}^3$, and hence its genus is $10$.

\begin{prop}
    The curve $X_0^+(243)$ is tetragonal over $\C$.
\end{prop}

\begin{proof}
    We compute the genus $g(X_0^+(243))=7$ and the Betti table. In particular, we get $\beta_{2,2}=9$. By \cite[Table 1.]{Schreyer1986}, this implies the existence of $g_6^2$. Now we use a similar argument as in \cite[Page 310, Case 3]{JeonPark05}. 

    Since there exists a $g_6^2$, there is a divisor $D$ such that $\deg D=6$ and $\ell(D)=3$. By definition, for this $D$ we have $\textup{Cliff}(D)=2$. We now apply Riemann-Roch theorem to get
    $$3-\ell(K-D)=\ell(D)-\ell(K-D)=6-7+1=0.$$
    Therefore, $\ell(K-D)=3$ and we have proven that $\textup{Cliff}(X_0^+(243))\leq2$. 
    
    The homogenous ideal of the canonical embedding of this curve is generated by quadrics (\texttt{Magma} function X0NQuotient(243,[243]) gives the canonical embedding for example). By \cite[Theorem 1.1]{JeonPark05}, this implies that $\textup{Cliff}(X_0^+(243))\geq2$ and we conclude that $\textup{Cliff}(X_0^+(243))=2$.

    Since $\textup{Cliff}(X_0^+(243))=\textup{Cliff}(D)=2$, from the definition of the Clifford dimension $\textup{CD}(X)$ we get that $\textup{CD}(X_0^+(243))\leq2$. Since $g(X_0^+(243))=7$, it is not a smooth plane curve of degree $d\geq5$ (a smooth plane curve of dimension $d$ has genus $(d-1)(d-2)/2$) and we have that $\textup{CD}(X_0^+(243))=1$.
    
    Let us now take a divisor $D'$ that computed the Clifford dimension. We have $\ell(D')=2$ and $\textup{Cliff}(D')=\textup{Cliff}(X_0^+(243))=2$. Therefore, from the definition of $\textup{Cliff}(D')$ we compute that $\deg D'=4$. We may assume that $D'>0$, otherwise we can take an effective divisor linearly equivalent to $D'$ (which exists because $\ell(D')\geq1$) and it will also compute the Clifford dimension of $X_0^+(243)$.

    There exists a non-constant morphism $f\in L(D')$ and its degree is at most $4$, otherwise $\textup{div}(f)+D'\ngeq 0$ (here we used that $D'$ is effective). Since $\textup{gon}_\C(X_0^+(243))\geq4$ by \cite{HasegawaShimura1999}, $f$ is the desired degree $4$ morphism from $X_0^+(243)$ to $\mathbb{P}^1$.
\end{proof}

\subsection{Proofs of the main theorems}

The curve $X_0^+(N)$ is hyperelliptic of genus $g\geq3$ if and only if \cite{FurumotoHasegawa1999} $$N\in\{60,66,85,92,94,104\}.$$
From now on, we will consider only the non-hyperelliptic curves of genus $g\geq3$.

\begin{proof}[Proof of \Cref{trigonalthm}]
    As was mentioned before, Hasegawa and Shimura \cite{HasegawaShimura1999} already solved the cases when the genus of $X_0^+(N)$ is not equal to $4$. For levels $N$ in the list for which the genus is equal to $4$, i.e. for 
    $$N\in\{84,88,93,115,116,129,135,137,147,155,159,215\},$$
    we used the \texttt{Magma} function Genus4GonalMap(C) to find the degree $3$ rational morphism to $\mathbb{P}^1$.

    Similarly to the function Genus6GonalMap(C) used in the proof of \Cref{genus6gonalmap}, this function always returns a morphism of degree $\leq3$ to $\mathbb{P}^1$ since all genus $4$ curves are hyperelliptic or $\C$-trigonal by \Cref{poonen} (v). However, the degree $3$ morphism can, in the general case, be defined over a quadratic field and this is, indeed, what happens with the curves in \Cref{CtrigonalQtetragonalthm}.
\end{proof}

\begin{proof}[Proof of \Cref{CtrigonalQtetragonalthm}]
    All these curves are of genus $4$ and therefore $\textup{gon}_\C X_0^+(N)=3$. We used \Cref{Fp_gonality} to prove that there are no degree $3$ rational maps to $\mathbb{P}^1$. Therefore, $\Q$-gonality must be equal to $4$ by \Cref{poonen} in these cases.
\end{proof}

\begin{proof}[Proof of \Cref{tetragonalthm}]
    The curves $X_0^+(N)$ of genus $g\leq4$ have been dealt with in Theorems \ref{trigonalthm} and \ref{CtrigonalQtetragonalthm}. We now consider the curves of genus $g\geq5$. We know from \cite{HasegawaShimura1999} that the curve $X_0^+(N)$ of genus $g\geq5$ is not trigonal over $\C$. 
    
    For levels $N$ in the list, in \Cref{degree4mapsection} we find a rational degree $4$ map to $\mathbb{P}^1$, meaning that $\textup{gon}_\Q X_0^+(N)=4$ in these cases. 
    
    For other levels $N$, we prove in \Cref{Fpsection} and \Cref{CSsection} that there are no rational degree $4$ maps to $\mathbb{P}^1$, and so $\textup{gon}_\Q X_0^+(N)>4$ in these cases. Moreover, in \Cref{bettisection}, we prove that $\textup{gon}_\C X_0^+(N)>4$ for levels $N\notin\{243,271\}$.
\end{proof}

\begin{cor}
    The $\Q$-gonality of $X_0(N)$ is equal to $8$ for 
    $$N\in\{193,194,207,224,229,241,257,281\}.$$
    For $N\in\{194,224,257,281\}$ the $\C$-gonality of $X_0(N)$ is also $8$.
\end{cor}

\begin{proof}
    The composition map $X_0(N)\to X_0^+(N)\to \mathbb{P}^1$ is a degree $8$ rational map and from \cite[Tables 1.,2.,3.]{NajmanOrlic22} we get that $\textup{gon}_\Q(X_0(N))>7$. For $N=193,207,229,241$ this follows from the bound on $\F_p$-gonality; codes for that can be found on 
    \begin{center}
        \url{https://github.com/orlic1/gonality_X0/tree/main/Fp_gonality}).
    \end{center}
    For $N=194,224,257,281$ we can use Castelnuovo-Severi inequality to prove it, meaning that we also get the lower bound on $\C$-gonality in these cases.
\end{proof}

\bibliographystyle{siam}
\bibliography{bibliography1}

\end{document}